\documentclass[amstex,12pt,reqno]{amsart}
\usepackage{amsmath,amsthm,amscd,amsfonts,amssymb,enumitem}
\usepackage[scr=rsfs]{mathalpha}
\usepackage{latexsym}
\usepackage{graphicx}
\usepackage{caption}
\usepackage{multirow}
\usepackage{lipsum} 
\usepackage{enumitem}
\usepackage[colorlinks=true, linkcolor=blue, anchorcolor=blue, citecolor=blue, filecolor=blue, urlcolor=blue]{hyperref}
\newtheorem{thm}{Theorem}[section]
\newtheorem{lemma}[thm]{Lemma}

\theoremstyle{definition}
\newtheorem{defin}[thm]{Definition}

\theoremstyle{remark}

\newcommand{\qedwhite}{\hfill \ensuremath{\Box}}
\renewenvironment{proof}{{\raggedright \bfseries Proof.}}{\qedwhite}

\newcommand{\floor}[1]{\left\lfloor {#1} \right\rfloor}
\newcommand{\ceil}[1]{\left\lceil {#1} \right\rceil}
\newcommand\numberthis{\addtocounter{equation}{1}\tag{\theequation}}
\numberwithin{equation}{section}

\begin{document}

\leftline{ \scriptsize \it  }
\title[]{Weighted approximation By Max-product Kantrovich type Exponential Sampling Series}


\maketitle

\begin{center}
    \bf  $^1${Satyaranjan Pradhan},
    \bf  $^2${Madan Mohan Soren}
    \vskip0.2in
    $^{1,2}${ Department of Mathematics, Berhampur University, Berhampur, Odisha, 760007, India} \\
    \verb"satyaranjanpradhan6312@gmail.com",
    \verb"mms.math@buodisha.edu.in "
\end{center}	


\begin{abstract}
    
In this study, we examine the convergence characteristics of the  Max-Product Kantrovich type exponential sampling series within the weighted space of log-uniformly continuous and bounded functions. The research focuses on deriving fundamental convergence results for the series and analyzing its asymptotic convergence behavior. The study estimates the rate of convergence using the weighted logarithmic modulus of continuity and establishes a quantitative Voronovskaja-type theorem offering insights into the asymptotic behavior of the series. Additionally, we present the example of kernel functions satisfying assumptions of the presented theory along with graphical demonstration and estimates of approximation error.\\
           
 
\end{abstract}

\section{Introduction}\label{section1}  
Sampling and reconstructing functions from discrete data points lie at the heart of many necessary scientific and engineering fields in approximation theory, with application to signal analysis, image processing \cite{apl1, apl2}, etc. Collectively, the Whittaker-Kotelnikov-Shannon achieved a significant breakthrough in sampling and reconstruction theory. They showed that the sample values of any band-limited signal $f$, i.e., whose Fourier transform $f$ is compactly supported, can recover the signal entirely (see \cite{srvbut}). It is known as widely as \textit{WKS sampling theorem}. This result was generalized to not necessarily band-limited signals by Butzer and Stens \cite{but1} by replacing the sinc function in the WKS sampling theorem with a more generalized kernel satisfying suitable assumptions. Since then, several questions have been posed, and various mathematicians have made significant advances in this direction, see \cite{butzer2, k2007, tam}. 

The problem of approximating functions given their exponentially-spaced sample values has been discussed in the work of Ostrowsky et al. \cite{ostrowsky}, Bertero and Pike \cite{bertero} and Gori \cite{gori}. They also provided a series representation for the class of Mellin band limited functions (defined in Section \ref{section2}) to handle exponentially spaced data. This reconstruction formula is referred to as the \textit{exponential sampling formula} and is defined as follows. For $h:\mathbb{R}_{+} \rightarrow \mathbb{C}$ and $l \in \mathbb{R},$ the exponential sampling formula is given by (see \cite{butzer3})
\begin{equation} \label{expformula}
\text(X_{l,P}\; h)(z):= \sum_{j =-\infty}^{\infty} lin_{\frac{l}{P}}(e^{j}z^{P}) \;h(e^{\frac{j}{P}}),
\end{equation}
where $lin_{l}(z)= \dfrac{z^{-l}}{2\pi i} \dfrac{z^{\pi i}-z^{-\pi i}}{\log l} = z^{-c} sinc(\log z)$ with continuous extension $lin_{l}(1)=1.$ Moreover, if $h$ is Mellin band-limited to $[-P,P],$ then $(X_{l,P} h)(z)= h (z)$ for each $ z \in \mathbb{R}_{+}.$
 
Scientific applications along with engineering applications commonly utilize exponentially spaced data to study Fraunhofer diffraction and perform photon correlation spectroscopy in polydispersity studies and neuron scattering together with radio astronomical research etc (see \cite{casasent,ostrowsky,bertero,gori}).Therefore exponential sampling applications have inspired a renewed focus on advancing and refining the original formula given in \eqref{expformula}. Butzer and Jansche \cite{butzer5} contributed a major advancement through their use of Mellin analysis tools which transformed the understanding of exponential sampling frameworks. Their approach established Mellin transform techniques as an optimal method for analyzing sampling and approximation problem associated with exponential spaced data.Mamedov started the initial groundwork that established Mellin transform theory which he first introduced in his paper \cite{mamedeo}. Butzer and his collaborators built upon this field through their works documented in \cite{butzer3,butzer5}. The works presented in \cite{bardaro1,bardaro9,bardaro2,bardaro3} with others contain substantial developments in Mellin analysis. The paper of Bardaro et al. in \cite{bardaro7} introduced an extended exponential sampling formula \eqref{expformula} to approximate functions beyond strict Mellin band-limited conditions. An extended exponential sampling formula uses appropriate kernel conditions to approximate log-continuous functions through their exponentially spaced data points. In \cite{bardaro7}, for $z \in \mathbb{R}_{+}$ , $\mathrm{m}>0, j\in Z \And h : R_{+} \to R $ the generalized exponential sampling series is defined as follows:
\begin{equation} \label{genexp}
(G_{\mathrm{m}}^{\varkappa}\; h)(z)= \sum_{j=- \infty}^{\infty} \varkappa(e^{-j} z^{\mathrm{m}}) \; h( e^{\frac{j}{\mathrm{m}}}),
\end{equation}

where $\varkappa$ is a kernel function satisfying the suitable assumptions such that the series \eqref{genexp} converges absolutely. Various approximation properties associated with the family of operators \eqref{genexp} can be observed in \cite{comboexp, bardaro11, bevi, diskant}. The works \cite{sn, self} deal with the approximation behavior of exponential sampling operators with artificial neural network structures. To approximate the Lebesgue integrable functions on $R_{+}$, the expansion \eqref{genexp} is not so convenient. To address this limitation, a Kantorovich-type modification of the exponential sampling series \eqref{genexp} was recently introduced in \cite{owndurr}. For $ z \in \mathbb{R}^{+}, j \in \mathbb{Z}$ and $\mathrm{m}>0,$ the Kantorovich  exponential sampling series of $f$ is defined by

\begin{equation} \label{kant}
(I_{\mathrm{m}}^{\varkappa } h)(z):= \sum_{j= - \infty}^{\infty} \varkappa(e^{-j} z^{\mathrm{m}})\  \mathrm{m} \int_{\frac{j}{\mathrm{m}}}^{\frac{j+1}{\mathrm{m}}} h(e^{v})\  dv,
\end{equation}

whenever the series (\ref{kant}) is absolutely convergent for any locally integrable function $ f: \mathbb{R}^{+} \rightarrow \mathbb{R}.$ 
Recent work by Coroianu and Gal introduced an interesting procedure to improve the order of approximation that can be attained by a family of linear operators by a so-called max-product method; see, for example, Coroianu and Gal (2010) \cite{CG2010}, Coroianu and Gal (2011) \cite{CG2011} and Coroianu and Gal (2012) \cite{CG2012}. In general, discrete linear operators are defined by finite sums or series concerning certain indexes. The max-product operators are defined by replacing the sums or the series with a maximum or a supremum computed over the same sets of indexes. The above procedure allows the conversion of linear operators into nonlinear ones, which can achieve a higher order of approximation concerning their linear counterparts.
\par

In the present paper, we investigate the approximation behaviors of  Max Product Kantrovich exponential sampling operators for functions belonging to a weighted space of functions. More precisely, for the above operators, the point-wise and uniform convergence are established together with the rate of convergence by involving the weighted modulus of continuity. Finally, we prove the quantitative Voronovskaja-type theorem for the operators.
\par
The paper is organized as follows: The Sections \ref{section1} and \ref{section2} are devoted to fundamentals of Max product Kantrovich type exponential sampling and some auxiliary results. Section \ref{section3} is dedicated to the well-definiteness of the operators between weighted spaces of functions, together with pointwise and uniform convergence. It derives the rate of convergence of the family of operators in terms of the weighted modulus of continuity. Furthermore, a pointwise convergence result in quantitative form is provided using a Voronovskaja-type theorem. In Section \ref{section4}, we present three examples of kernel functions and numerical experiments that demonstrate the performance of proposed operators for various test functions and graphical visualizations of their approximation and convergence behavior. Finally, the paper concludes with a summary in Section \ref{section5}.
\section{Preliminaries}\label{section2} 
Throughout this paper, we consider the interval $\mathscr{I} = [a,b]$ to be any compact subset of $\mathbb{R}_{+}$, where $\mathbb{R}_{+}$ denotes the set of positive real numbers. Let $\mathscr{C(I)}$ denote the space of uniformly continuous function on $\mathscr{I}$, equipped with standard the supremum norm  
\begin{equation*}
  \|h\|_{\infty} := \sup_{z\in \mathscr{I}} |h(z)|.
\end{equation*}

 We also introduce $\mathscr{C_{+}(I)}$, the subspace of non negative function in $\mathscr{C(I)}$.
A function $h:\mathbb{R}_{+} \rightarrow \mathbb{R}$ is defined to be log uniformly continuous function if, for any given $\epsilon > 0$, there exist $\rho > 0 $ such that $|h(z_{1}) - h(z_{2})| < \epsilon $ whenever $|\log(z_{1}) -\log(z_{2})| \leq \rho,$ for all $z_{1},z_{2} \in \mathbb{R}_{+} .$

Now we introduce the space  $\mathscr{U}(\mathscr{I})$ consisting of all log uniformly continuous functions defined on $\mathscr{I}$. This space forms a completely normed linear space equipped with the maximum norm. The subspace  $\mathscr{U_{+}}(\mathrm{I})$ comprises all non-negative function with in $\mathscr{U}(\mathscr{I})$. Furthermore, we denote
\begin{itemize}[label={}]
    \item $\mathscr{B}(\mathbb{R}_{+}) := $ The space of all bounded functions on $\mathbb{R_{+}}.$
    \item $\mathscr{V}(\mathbb{R_{+}}):= $ The space of all log uniformly continuous and bounded functions defined on $\mathbb{R_{+}}.$
    \item $\mathscr{V}_{+}(\mathbb{R_{+}}):=$ The space of all non negative  log uniformly continuous and bounded function defined on $\mathbb{R_{+}}.$
\end{itemize}
We are now going to introduce the weighted function space for a particular weight function $\mathrm{w}( z) = \frac{1}{1+\log^{2}z}$.
We denote the weighted function space $\mathscr{B}^{\mathrm{w}}(\mathbb{R}_{+}) :=\{ h:\mathbb{R_{+} }\to \mathbb{R} :  \exists K > 0$ such that $\mathrm{w}(z)|h(z)|\leq K $  $\forall z\in \mathbb{R_{+}}\}$ associated with $\mathscr{B}(\mathbb{R}_{+})$  and we define the weighted space associated with $\mathscr{V}(\mathbb{R_{+}})$ as $\mathscr{V}^{\mathrm{w}}(\mathrm{R}_{+}) := \{h:\mathbb{R_{+} }\to \mathbb{R} : \mathrm{w} h  \in \mathscr{V}(\mathbb{R_{+}}) \}$ and $\mathscr{V}^{\mathrm{w}}_{+}(\mathbb{R_{+}})$, the subclass of $\mathscr{V}^{\mathrm{w}}(\mathrm{R}_{+})$ which is non negative.

It is interesting to mention that the linear space of functions  $\mathscr{V}^{\mathrm{w}}(\mathbb{R}_{+})$  is normed linear space with the norm 
\begin{equation*} 
\| h\|_{\mathrm{w}} = \sup_{z>0} \mathbf{\mathrm{w}}(z)|h(z)|.
\end{equation*}

\begin{defin}\label{def1}
    The weighted logarithmic modulus of continuity for $h \in \mathscr{V}^{\mathrm{w}}(\mathbb{R}_{+})$ and $\rho > 0$ is considered as 
    \begin{equation*}
    \Upsilon(h,\rho)= \sup_{|\log u| \leq \rho , z > 0} \frac{|f(uz)-f(z)|}{(1+\log^{2}z) (1+\log^{2}u)}.
    \end{equation*}  
\end{defin}

\begin{lemma}\label{lma4}
    The weighted logarithmic modulus of continuity $\Upsilon(h,\rho)$ has the following fundamental Properties:
    \begin{enumerate}
	\item $\Upsilon(h,\rho)$ is a monotonically increasing function of $\rho$.
	\item  For $h \in \mathscr{V}^{\mathrm{w}}(\mathbb{R}_{+}))$ the quantity $\Upsilon(h,\rho)$ is finite.
	\item For all $ h \in \mathscr{V}^{\mathrm{w}}(\mathbb{R}_{+})$  and each $ \lambda \in \mathbb{R}_{+},$
	$$\Upsilon(h,\lambda \rho)\leq 2 (1+\lambda)^{3} (1+\rho^{2})\Upsilon (h,\rho).$$
	\item For all $h \in \mathscr{V}^{\mathrm{w}}(\mathbb{R}_{+})$ and $s,z >0$
	$$|(f(s)- f(z)|\leq 16 (1+\rho^{2})^{2} (1+\log^{2}z) (1+ \frac{|\log s-\log z|^{5}}{\rho^{5}}) \Upsilon(h,\rho).$$
	\item For $h \in \mathscr{V}^{\mathrm{w}}(\mathbb{R}_{+}),$ we have $\lim_{\rho \to 0} \Upsilon(h,\rho) = 0.$
    \end{enumerate}  
\end{lemma}
Now, we introduce the notation that will be used to establish the analysis of the max-product sampling operators.
\par
Given any index set $ \Lambda \subseteq \mathbb{Z}$, we define $$ \bigvee_{\mathrm{j}\in \Lambda } \mathscr{T}_{j} := \sup\{ \mathscr{T}_{j}  :  \mathrm{j} \in \Lambda\} .$$
If $\Lambda $ is finite, then $$ \bigvee_{\mathrm{j}\in \Lambda } \mathscr{T}_{j} = \max_{\mathrm{j}\in \Lambda } \mathscr{T}_{j}\; .$$
\par
Again, we introduce the functions used as kernels of the max product generalized exponential sampling operator. In the subsequent analysis, we define a kernel as a function  $\varkappa :\mathbb{R}_{+} \rightarrow \mathbb{R} $ that is both bounded and measurable, and satisfies the following conditions.
\begin{enumerate}
    \item[$(\varkappa_{1}) :$] There exists $\kappa >0$, such that the discrete absolute moment of order $\kappa $ is finite for $\kappa = 2$ i.e. 
    $$m_{\kappa}(\varkappa)= \sup_{s\in \mathbb{R_{+}}} \bigvee_{j\in z} |\varkappa(e^{-j}s)||j-\log s|^{\kappa} \;  < \infty\quad \text{for}\,\,\kappa =2. $$
    \item[$(\varkappa_{2}):$] there holds $$\inf_{z \in [1,e]} \varkappa(z) =: \zeta_{z}.$$
\end{enumerate}
	
\begin{lemma}\label{lma1}
    Let $\varkappa$ be a bounded function satisfying $(\varkappa_{1})$ with $\kappa > 0$. Then $m_{\mu}(\varkappa) < \infty,$ for every $ 0 \leq \mu \leq \kappa.$   
\end{lemma}
\begin{proof}
    Let $ 0 \leq \mu \leq \kappa $ be fixed. For every $s \in \mathbb{R_{+}},$ we have
    \begin{align*}
        \bigvee_{j\in \mathbb{Z}} |\varkappa(e^{-j}s)||j-\log (s)|^{\mu} \leq & \bigvee_{\substack{j \in \mathbb{Z} \\ |j-\log(s)| \leq 1}} |\varkappa(e^{-j}s)||j-\log s|^{\mu} \\ & + \bigvee_{\substack{j \in \mathbb{Z} \\ |j-\log(s)| \leq 1}} |\varkappa(e^{-j}s)||j-\log s|^{\mu}\\
        \leq &  m_{0}(\varkappa) +  m_{\mu}(\varkappa)\\
        \leq & \infty.
    \end{align*}
\end{proof}
\begin{lemma}\label{lma2}
    Let $\varkappa : \mathbb{R_{+}}\rightarrow \mathbb{R}$ be a kernel satisfying $ m_{\kappa}(\varkappa) < \infty, \kappa > 0$. Then, for every $\rho >0$, we have \[ \bigvee_{\substack{j\in \mathbb{Z}\\|j-\log(s)| > \mathrm{m}\rho }} |\varkappa(e^{-j}s)| = \mathcal{O}(\mathrm{m^{-\nu}}), \; \text{as $ \mathrm{m} \rightarrow \infty $}\] uniformly with respect to $s \in \mathbb{R_{+}}$.
\end{lemma}
\begin{proof}
    Let $s \in \mathbb{R_{+}}$ be fixed. Then, we get
    \begin{align*}
        \bigvee_{\substack{j\in \mathbb{Z}\\|j-\log(s)| > \mathrm{m}\rho }}|\varkappa(e^{-j}s)| = &\bigvee_{\substack{j\in \mathbb{Z}\\|j-\log(s)| > \mathrm{m}\rho }} |\varkappa(e^{-j}s)| \frac{|j-\log(s)|^{\mu}}{|j-\log(s)|^{\mu}}\\
        \leq & \frac{1}{(\mathrm{m}\rho)^{\mu}}\bigvee_{\substack{j\in \mathbb{Z}\\|j-\log(s)| > \mathrm{m}\rho }}|\varkappa(e^{-j}s)| |j-\log(s)|^{\mu}\\
        \leq & \frac{m_{\mu}(\varkappa)}{( \mathrm{m} \rho)^{\mu}} < +\infty.
    \end{align*} 
    Hence, the desired result follows.
\end{proof}
\begin{lemma}\label{lma3}
    Let $\varkappa : \mathbb{R_{+}}\rightarrow \mathbb{R}$ be a kernel satisfying the condition $(\varkappa_{2})$. Then, we have the following
    \begin{enumerate}
        \item For any $z \in \mathscr{I}$,  $$\bigvee_{j\in \mathbb{I}_{\mathrm{m}} } |\varkappa(e^{-j}z^{\mathrm{m}})| \geq \zeta_{z},$$ where $\mathbb{I}_{\mathrm{m}} = \{ j \in \mathbb{Z} : j = \ceil{\mathrm{m}\log a}.....,\floor{\mathrm{m}\log b} \}.$
        \item For any $z \in \mathbb{R_{+}}$,  $$\bigvee_{j\in \mathbb{Z}} |\varkappa(e^{-j}z^{\mathrm{m}})| \geq \zeta_{z}.$$
    \end{enumerate}  
\end{lemma}
\begin{proof}
   We prove the first part of the lemma, as the second part follows analogously. Let $z\in [a,b]$ be fixed. Then, there exist at least a $ j_{1} \in \mathbb{I}_{\mathrm{m}}$ such that $e^{-j}z^{\mathrm{m}} \in [1,e] $.
   
    Thus, we obtain $$ \bigvee_{j\in \mathbb{Z}} \varkappa(e^{-j}z^{\mathrm{m}}) \geq \varkappa(e^{-{j}_{1}}z^{\mathrm{w}}) \geq \zeta_{z}.$$
    Thus, the proof is complete.
\end{proof}
 
\begin{defin}
    Let $ h : \mathscr{I} \to \mathbb{R} $ be any locally integrable function on $\mathscr{I}$. Let $\varkappa$ be a kernel functions satisfying $\bigvee\limits_{j\in \mathbb{I}_{w}} \varkappa(e^{-j}z^{\mathrm{m}}) \neq 0 , \forall  z \in \mathscr{I}$. Then the max product exponential sampling operators are defined by $$ \mathscr{M}^{\varkappa}_{\mathrm{m}}(h,z) := \frac{\bigvee\limits_{j\in \mathbb{I}_{\mathrm{m}} } \varkappa(e^{-j}z^{\mathrm{m}})\; \mathrm{m} \int_{\frac{j}{\mathrm{m}}}^{\frac{j+1}{\mathrm{m}}} h(e^{v}) \;dv}{\bigvee\limits_{j\in \mathbb{I}_{\mathrm{m}} } \varkappa(e^{-j}z^{\mathrm{m}})},$$ where $\mathbb{I}_{\mathrm{m}} = \{ j \in \mathbb{Z} : j = \ceil{\mathrm{m}\log a}.....,\floor{\mathrm{m}\log b}-1\}.$
\end{defin}

\begin{lemma}\label{lma5}
    Let $h,g \in \mathscr{B_{+}}(\mathscr{I}) $ be any locally integrable functions on $\mathscr{I}$ . Then we have
    \begin{enumerate}
        \item[(1)] If  $ h(z) \leq g(z), \; then \;\mathscr{M}^{\varkappa}_{\mathrm{m}}(h,z) \leq \mathscr{M}^{\varkappa}_{\mathrm{m}}(g,z) $.
        \item[(2)] $\mathscr{M}^{\varkappa}_{\mathrm{m}}(h+g,z) \leq \mathscr{M}^{\varkappa}_{\mathrm{m}}(h,z)+ \mathscr{M}^{\varkappa}_{\mathrm{m}}(g,z)$.
        \item[(3)] $|\mathscr{M}^{\varkappa}_{\mathrm{m}}(h,z)- \mathscr{M}^{\varkappa}_{\mathrm{m}}(g,z)| \leq \mathscr{M}^{\varkappa}_{\mathrm{m}}(|h-g|,z) $, for all $z\in \mathscr{I}$.
        \item[(4)]$\mathscr{M}^{\varkappa}_{\mathrm{m}}(\lambda h,z) = \lambda \hspace{0.1in} \mathscr{M}^{\varkappa}_{\mathrm{m}}(h,z)$, for every $\lambda > 0$.
    \end{enumerate}
\end{lemma}
\begin{proof}
    Parts $(1), (2)$ and $(4)$ follow directly from the definitions of operators $\mathscr{M}^{\varkappa}_{\mathrm{m}}$. To prove part $(3)$, we observe that $ h(z) \leq |h(z) - g(z)| + g(z) $ and $ g(z) \leq |g(z) - h(z)| + h(z) $. Applying properties $(1)$ and $(2)$, we obtain 
    \begin{align*}
        \mathscr{M}^{\varkappa}_{\mathrm{m}}(h,z) \leq \mathscr{M}^{\varkappa}_{\mathrm{m}}(|h-g|,z)+ \mathscr{M}^{\varkappa}_{\mathrm{m}}(g,z) \And  \mathscr{M}^{\varkappa}_{\mathrm{m}}(g,z) \leq \mathscr{M}^{\varkappa}_{\mathrm{m}}(|g-h|,z)+ \mathscr{M}^{\varkappa}_{\mathrm{m}}(h,z).
    \end{align*}
    Combining these inequalities, we conclude that $$|\mathscr{M}^{\varkappa}_{\mathrm{m}}(h,z)- \mathscr{M}^{\varkappa}_{\mathrm{m}}(g,z)| \leq \mathscr{M}^{\varkappa}_{\mathrm{m}}(|h-g|,z), \text{\hspace{0.1in} $ \forall z\in \mathscr{I}$}.$$
\end{proof}
\section{Results for Max-Product Kantrovich Exponential Sampling Series} \label{section3}
The first main result shows that the operators $\mathscr{M}^{\varkappa}_{\mathrm{m}}$ are well defined on logarithmic weighted space of functions. We need the following preliminary propositions.
\begin{thm}
    Let $\varkappa$ be a kernel satisfying the assumptions $(\varkappa_{1}), (\varkappa_{2})$ for $\kappa = 2$. Then for a fixed $\mathrm{m} >0$ the operator $\mathscr{M}^{\varkappa}_{\mathrm{m}}$ is a operator from $\mathscr{B}^{\mathrm{w}}(\mathbb{R}_{+}) \to \mathscr{B}^{\mathrm{w}}(\mathbb{R}_{+})$ and its operator norm turns out to be
    \begin{align*}
       \|\mathscr{M}^{\varkappa}_{\mathrm{m}}\|_{\mathscr{B}^{\mathrm{w}}(\mathbb{R}_{+}) \to \mathscr{B}^{\mathrm{w}}(\mathbb{R}_{+})} \leq  \frac{1}{\zeta^{2}_{z}} \left[m_{0}(\varkappa)(1+\frac{1}{\mathrm{m}}+\frac{1}{3\mathrm{m}^2})+ (\frac{1}{\mathrm{m}^2} +\frac{1}{\mathrm{m}})m_{1}(\varkappa)+ \frac{1}{\mathrm{m}^2}m_{2}(\varkappa)\right].
    \end{align*}
\end{thm}
\begin{proof}
    Let us fix $\mathrm{m} > 0$. Using the definition of new form $\mathscr{M}^{\varkappa}_{\mathrm{m}}(h,z)$, we have 
    $$ |\mathscr{M}^{\varkappa}_{\mathrm{m}}| \leq \frac{1}{\zeta_{z}}\bigvee_{j\in \mathbb{Z}}   | \varkappa(e^{-j}z^{\mathrm{m}})|m \int_{\frac{j}{\mathrm{m}}}^{\frac{j+1}{\mathrm{m}}}  | f(e^{v})| dv .$$ 
    Moreover, since $h \in  \mathscr{B}^{\mathrm{w}}(\mathbb{R}_{+}) $ and recalling $\Phi = \frac{1}{\mathrm{w}}$, we obtain
    \begin{align*}
	|\mathscr{M}^{\varkappa}_{\mathrm{m}}(h,z)| 
        \leq & \frac{1}{\zeta_{z}} \bigvee_{j\in \mathbb{Z}} 
               |\varkappa(e^{-j}z^{\mathrm{m}})| \; m \int_{\frac{j}{\mathrm{m}}}^{\frac{j+1}{\mathrm{m}}} |h(e^{v})| dv\\ 
	  = & \frac{1}{\zeta_{z}} \bigvee_{j\in \mathbb{Z}} 
               |\varkappa(e^{-j}z^{\mathrm{m}})| \; m \int_{\frac{j}{\mathrm{m}}}^{\frac{j+1}{\mathrm{m}}} |h(e^{v})| \frac{1+\log^2(e^v)}{1+\log^2(e^v)} dv\\ 
	  = & \frac{\|h\|_{\mathrm{w}}}{\zeta_{z}}  \bigvee_{j\in 
            \mathbb{Z}} |\varkappa(e^{-j}z^{\mathrm{m}})| \; \mathrm{m} \int_{\frac{j}{\mathrm{m}}}^{\frac{j+1}{\mathrm{m}}}  (1+\log^2(e^v)) dv\\
	  = & \frac{\|h\|_{\mathrm{w}}}{\zeta_{z}}  \bigvee_{j\in 
            \mathbb{Z}} |\varkappa(e^{-j}z^{\mathrm{m}})| \; \mathrm{m} \int_{\frac{j}{\mathrm{m}}}^{\frac{j+1}{\mathrm{m}}} (1+(\log(e^v))^2) dv\\
	  = & \frac{\|h\|_{\mathrm{w}}}{\zeta_{z}}  \bigvee_{j\in 
            \mathbb{Z}} |\varkappa(e^{-j}z^{\mathrm{m}})| \; \mathrm{m} \int_{\frac{j}{\mathrm{m}}}^{\frac{j+1}{\mathrm{m}}}  (1+v^2) dv \\
		= & \frac{\|h\|_{\mathrm{w}}}{\zeta_{z}}  \bigvee_{j\in 
            \mathbb{Z}} |\varkappa(e^{-j}z^{\mathrm{m}})| \; \mathrm{m} \; \left[v+\frac{v^3}{3}\right]_{\frac{j}{\mathrm{m}}}^{\frac{j+1}{\mathrm{m}}}  \\
		= &  \frac{\|h\|_{\mathrm{w}}}{\zeta_{z}}  \bigvee_{j\in 
            \mathbb{Z}} |\varkappa(e^{-j}z^{\mathrm{m}})| \; \mathrm{m} \; \left[\frac{1}{\mathrm{m}}+ \frac{1}{3} \left( \left(\frac{j+1}{\mathrm{m}}\right)^3- \left(\frac{j}{\mathrm{m}}\right)^3\right) \right]\\
		= & \frac{\|h\|_{\mathrm{w}}}{\zeta_{z}}  \bigvee_{j\in 
            \mathbb{Z}} |\varkappa(e^{-j}z^{\mathrm{m}})| \; \mathrm{m} \; \left[\frac{1}{\mathrm{m}}+ \frac{1}{3\mathrm{m}}  \left(\frac{3j^2}{\mathrm{m}^2}+\frac{3j}{\mathrm{m}^2}+\frac{1}{\mathrm{m}^2} \right) \right]\\
		= & \frac{\|h\|_{\mathrm{w}}}{\zeta_{z}}  \bigvee_{j\in 
            \mathbb{Z}} |\varkappa(e^{-j}z^{\mathrm{m}})| \;  \left[1+ \frac{1}{3}  \left(\frac{3j^2}{\mathrm{m}^2}+\frac{3j}{\mathrm{m}^2}+\frac{1}{\mathrm{m}^2} \right) \right]\\
		= & \frac{\|h\|_{\mathrm{w}}}{\zeta_{z}}  \bigvee_{j\in 
            \mathbb{Z}} |\varkappa(e^{-j}z^{\mathrm{m}})| \;   \left[1+ \frac{1}{3\mathrm{m}^2} + \frac{j^2}{\mathrm{m}^2}+\frac{j}{\mathrm{m}^2} \right]\\
		= & \frac{\|h\|_{\mathrm{w}}}{\zeta_{z}}  \bigvee_{j\in 
            \mathbb{Z}} |\varkappa(e^{-j}z^{\mathrm{m}})|  \\ &  \left[1+ \frac{1}{3\mathrm{m}^2} + \left(\frac{j}{\mathrm{m}} - \log z \right)^{2} + \log^2 z+2\log z\left(\frac{j}{\mathrm{m}} - \log z\right) + \frac{1}{\mathrm{m}} \left(\frac{j}{\mathrm{m}} - \log z \right)+ \frac{\log z}{\mathrm{m}} \right] \\
	|\mathscr{M}^{\varkappa}_{\mathrm{m}}(h,z)|
	\leq & \frac{\|h\|_{\mathrm{w}}}{\zeta_{z}}  \bigvee_{j\in 
            \mathbb{Z}} |\varkappa(e^{-j}z^{\mathrm{m}})| \left(1+\log^{2} z\right) \\ &  \left[1+ \frac{1}{3\mathrm{m}^2} + \left( \left| \frac{j-\mathrm{m}\log{} z}{\mathrm{m}}  \right|  \right)^{2} +\left( \left| \frac{j-\mathrm{m}\log{} z}{\mathrm{m}}  \right| \right) + \frac{1}{\mathrm{m}} \left( \left| \frac{j-\mathrm{m}\log z}{\mathrm{m}}  \right| \right)+ \frac{1}{\mathrm{m}} \right] \\
	= & \frac{\|h\|_{\mathrm{w}}}{\zeta_{z}}  \bigvee_{j\in 
            \mathbb{Z}} |\varkappa(e^{-j}z^{\mathrm{m}})|  \left(1+\log^{2} z\right) \\ &  \left[1+\frac{1}{\mathrm{m}}+ \frac{1}{3\mathrm{m}^2} +  \frac{\left|j-\mathrm{m}\log{} z \right|^{2}}{\mathrm{m}^2}  +  \frac{\left|j-\mathrm{m}\log{} z \right|}{\mathrm{m}}  +   \frac{\left|j-\mathrm{m}\log{} z \right|}{\mathrm{m}^2}  \right] \\
	= & \frac{\|h\|_{\mathrm{w}}}{\zeta_{z}}  \left(1+\log^{2} 
            (z)\right)  \left[ \left(1+\frac{1}{\mathrm{m}}+ \frac{1}{3\mathrm{m}^2} \right) m_0(\varkappa)+ \left( \frac{1}{\mathrm{m}}  +   \frac{1}{\mathrm{m}^2} \right) m_1(\varkappa) + \left( \frac{1}{\mathrm{m}^2} \right) m_2(\varkappa)   \right],\\
    \end{align*}
    
    which implies for every $z\in \mathbb{R}^{+}$ that 
    $$ \frac{|\mathscr{M}^{\varkappa}_{\mathrm{m}}(h,z)|}{\left(1+\log^{2} z\right)} \leq \frac{\|h\|_{\mathrm{w}}}{\zeta_{z}}  \left[ \left(1+\frac{1}{\mathrm{m}}+ \frac{1}{3\mathrm{m}^2} \right) m_0(\varkappa)+ \left( \frac{1}{\mathrm{m}}  +   \frac{1}{\mathrm{m}^2} \right) m_1(\varkappa) + \left( \frac{1}{\mathrm{m}^2} \right) m_2(\varkappa)   \right] .$$
    Since the assumption $m_2(\varkappa) \leq + \infty$ implies  $m_i(\varkappa) \leq + \infty$ for $i= 0,1,$ we deduce $$ \| \mathscr{M}^{\varkappa}_{\mathrm{m}} h \|_{\mathrm{w}} \le +\infty, ~~\text{i.e.,}~~ \mathscr{M}^{\varkappa}_{\mathrm{m}}h \in \mathscr{B}^{\mathrm{w}}(\mathbb{R}_{+}).$$
  
     On the other hand taking supremum over $z\in \mathbf{R}_{+}$ in 
   $$ \frac{|\mathscr{M}^{\varkappa}_{\mathrm{m}}(h,z)|}{\left(1+\log^{2} z\right)} \leq \frac{\|h\|_{\mathrm{w}}}{\zeta_{z}}  \left[ \left(1+\frac{1}{\mathrm{m}}+ \frac{1}{3\mathrm{m}^2} \right) m_0(\varkappa) + \left( \frac{1}{\mathrm{m}^2} \right) m_2(\varkappa)  + \left( \frac{1}{\mathrm{m}}  +   \frac{1}{\mathrm{m}^2} \right) m_1(\varkappa) \right] $$
    and supremum with respect to $h \in \mathscr{B}^{\mathrm{w}}(\mathbb{R}_{+})$ with $\| h\|_{\mathrm{w}} \leq 1$, we have 
    $$ \| \mathscr{M}^{\varkappa}_{\mathrm{m}}\|_{\mathscr{B}^{\mathrm{w}}(\mathbb{R}_{+}) \to \mathscr{B}^{\mathrm{w}}(\mathbb{R}_{+})} \leq \frac{1}{\zeta_{z}}\left[ \left(1+\frac{1}{\mathrm{m}}+ \frac{1}{3\mathrm{m}^2} \right) m_0(\varkappa) + \left( \frac{1}{\mathrm{m}}  +   \frac{1}{\mathrm{m}^2} \right) m_1(\varkappa) + \left( \frac{1}{\mathrm{m}^2} \right) m_2(\varkappa)  \right].$$
		
\end{proof}
	
\begin{thm}
Let $\varkappa$ be a kernel satisfying $\varkappa_{1} \And{ \varkappa_{2}}$ for $\kappa=2$. Let $h:\mathbb{R_{+}} \to\mathbb{R_{+}}$ be a bounded function. Then $$ \lim_{\mathrm{m} \to \infty } \mathscr{M}^{\varkappa}_{\mathrm{m}}(h,z) = h(z) \; \text{holds at each log-continuity point $z \in \mathbb{R_{+}}$}. $$
Moreover, if $h \in \mathscr{V}^{\mathrm{w}}_{+}(\mathbb{R_{+}})$, then $$\lim_{\mathrm{m} \to \infty } \|\mathscr{M}^{\varkappa}_{\mathrm{m}}(h,z) - h(z)\|_{\mathrm{w}} = 0. $$
\end{thm}

\begin{proof}
By the definition of the Max-product Kantrovich exponential sampling operator, we write 
\begin{align*}
    |\mathscr{M}^{\varkappa}_{\mathrm{m}}(h,z) - h(z)| \leq  & |\mathscr{M}^{\varkappa}_{\mathrm{m}}(h,z) - h(z) \mathscr{M}^{\varkappa}_{\mathrm{m}}(1,z)| + |h(z) \mathscr{M}^{\varkappa}_{\mathrm{m}}(1,z) - h(z)|\\
    \leq & |\mathscr{M}^{\varkappa}_{\mathrm{m}}(h,z) - h(z) \mathscr{M}^{\varkappa}_{\mathrm{m}}(1,z)| + |h(z)\; (\mathscr{M}^{\varkappa}_{\mathrm{m}}(1,z) - 1)|.
\end{align*}
We now define $1:\mathbb{R_{+}} \to \mathbb{R_{+}}$ by $1(z)= z$ \text{and} $h_{z}: \mathbb{R_{+}} \to \mathbb{R_{+}}$ by $h_{z}(u) = h(z) $ for all $u\in \mathbb{R_{+}}$. Since $\mathscr{M}^{\varkappa}_{\mathrm{m}}(1,z)=1$, we have $$|\mathscr{M}^{\varkappa}_{\mathrm{m}}(h,z) - h(z)| \leq \mathscr{M}^{\varkappa}_{\mathrm{m}}(|h-h_{z}|,z).$$

Using the definition of operator and $\varkappa_{2}$, we get
\begin{align*}
     \mathscr{M}^{\varkappa}_{\mathrm{m}}(|h-h_{z}|,z) \leq \frac{1}{\zeta_{z}} \bigvee_{j\in\mathbb{Z}} | \varkappa(e^{-j}z^{\mathrm{m}})|\; \mathrm{m} \int_{\frac{j}{\mathrm{m}}}^{\frac{j+1}{\mathrm{m}}}|h(e^{v})- h(z)|\; dv .
\end{align*}
 Now for all $z \in  \mathbb{R_{+}} \And{\mathrm{m}} > 0 $, we obtain
 \begin{align*}
     h(e^{v}) - h(z) = &  h(e^{v}) - \frac{\mathrm{w}(e^{v}) h(e^{v})}{\mathrm{w}(e^{v})}  + \frac{\mathrm{w}(e^{v}) f(e^{v})}{\mathrm{w}(e^{v})}  - h(z)\\
     = & \mathrm{w}(e^{v}) h(e^{v}) \left( \frac{1}{\mathrm{w}(e^{v})} - \frac{1}{\mathrm{w}(z)}\right) + \frac{1}{\mathrm{w}(z)} \left( \mathrm{w}(e^{v}) h(e^{v}) - \mathrm{w}(z) h(z)\right ).
 \end{align*}
By using the above inequality, we can write as follows
\begin{align*}
    \Rightarrow |\mathscr{M}^{\varkappa}_{\mathrm{m}}(|h-h_{z}|,z)|
    \leq & \frac{1}{\zeta_{z}}\bigvee_{j\in \mathbb{Z}} |\varkappa(e^{-j}z^{\mathrm{m}})| 
          \mathrm{m} \int_{\frac{j}{\mathrm{m}}}^{\frac{j+1}{\mathrm{m}}} \left | h(e^{v}) - h(z) \right| dv \\
     = & \frac{1}{\zeta_{z}}\bigvee_{j\in \mathbb{Z}} |\varkappa(e^{-j}z^{\mathrm{m}})| 
           \mathrm{m} \int_{\frac{j}{\mathrm{m}}}^{\frac{j+1}{\mathrm{m}}} \mathrm{w}(e^{v})\; |h(e^{v})| \left| \frac{1}{\mathrm{w}(e^{v})}-\frac{1}{\mathrm{w}(z)}\right|  \\ & + \frac{1}{\mathrm{w}(z)} \left|\mathrm{w}(e^{v})h(e^{v})-\mathrm{w}(z)h(z)\right|  dv \\
    \leq & \frac{1}{\zeta_{z}}\bigvee_{j\in\mathbb{Z}} |\varkappa(e^{- j}z^{\mathrm{m}})| 
           \;\; \|h\|_{\mathrm{w}} \; \; \mathrm{m}\int_{\frac{j}{\mathrm{m}}}^{\frac{j+1}{\mathrm{m}}} |v^{2} - \log^{2}z| dv  \\ & +\frac{1}{\zeta_{z} \mathrm{w}(z)}\bigvee_{j\in \mathbb{Z}} |\varkappa(e^{-j}z^{\mathrm{m}})|  \; \mathrm{m}\int_{\frac{j}{\mathrm{m}}}^{\frac{j+1}{\mathrm{m}}} |\mathrm{w}(e^{v})h(e^{v})-\mathrm{w}(z)h(z)| dv\\
    = &    I_{1} + I_{2}. 
\end{align*}
Let us consider $I_{1}$. By a straight forward computation, we obtain
\begin{align*}
    I_{1} 
    = &    \frac{1}{\zeta_{z}}\bigvee_{j\in\mathbb{Z}} |\varkappa(e^{- j}z^{\mathrm{m}})| 
           \;\; \|h\|_{\mathrm{w}} \; \; \mathrm{m}\int_{\frac{j}{\mathrm{m}}}^{\frac{j+1}{\mathrm{m}}} |v^{2} - \log^{2}z| \; dv\\ 
    \leq & \frac{\|h\|_{\mathrm{w}}}{\zeta_{z}}\bigvee_{j\in 
           \mathbb{Z}} |\varkappa(e^{-j}z^{\mathrm{m}})|\;  \; \mathrm{m}\int_{\frac{j}{\mathrm{m}}}^{\frac{j+1}{\mathrm{m}}} \left[ (v-\log z)^2 + 2 |\log z| |v-\log z| \right] dv \\
    = &   \frac{\|h\|_{\mathrm{w}}}{\zeta_{z}}\bigvee_{j\in \mathbb{Z}} |\varkappa(e^{- 
          j}z^{\mathrm{m}})|\;  \; \mathrm{m} \left[\frac{1}{3} \left| \left( \frac{j+1}{\mathrm{m}}-\log z\right)^{3} - \left( \frac{j}{\mathrm{m}}-\log z\right)^{3} \right| + 2 |\log z| \int_{\frac{j}{\mathrm{m}}}^{\frac{j+1}{\mathrm{m}}} |v-\log z| \; dv \right]\\
    \leq &  \frac{\|h\|_{\mathrm{w}}}{\zeta_{z}}\bigvee_{j\in 
         \mathbb{Z}} |\varkappa(e^{-j}z^{\mathrm{m}})| \; \frac{\mathrm{m}}{3} \left[ \left( \frac{j+1}{\mathrm{m}} -\log z\right) - \left( \frac{j}{\mathrm{m}}-\log z\right) \right]^{3} \\ & +\frac{3}{\mathrm{m}} \left(\frac{j+1}{\mathrm{m}}-\log z\right)\left(\frac{j}{\mathrm{m}}-\log z\right)   + \frac{\|h\|_{\mathrm{w}}}{\zeta_{z}}\bigvee_{j\in 
         \mathbb{Z}} |\varkappa(e^{-j}z^{\mathrm{m}})| 2\; \mathrm{m} |\log z| \int_{\frac{j}{\mathrm{m}}}^{\frac{j+1}{\mathrm{m}}} |v-\log z| dv \\ 
    = &  \frac{\|h\|_{\mathrm{w}}}{\zeta_{z}}\bigvee_{j\in 
         \mathbb{Z}} |\varkappa(e^{-j}z^{\mathrm{m}})| \;  \; \frac{\mathrm{m}}{3} \left[ \left( \frac{1}{\mathrm{m}^3}\right) +\frac{3}{\mathrm{m}} \left(\frac{j+1}{\mathrm{m}}-\log z\right)\left(\frac{j}{\mathrm{m}}-\log z\right)\right] \\ &  + \frac{\|h\|_{\mathrm{w}}}{\zeta_{z}}\bigvee_{j\in 
         \mathbb{Z}} |\varkappa(e^{-j}z^{\mathrm{m}})| 2\; \mathrm{m} \; |\log z| \int_{\frac{j}{\mathrm{m}}}^{\frac{j+1}{\mathrm{m}}}  |v-\frac{j}{\mathrm{m}}+\frac{j}{\mathrm{m}}-\log z| dv \\ 
    \leq & \frac{\|h\|_{\mathrm{w}}}{\zeta_{z}}\bigvee_{j\in 
         \mathbb{Z}} |\varkappa(e^{-j}z^{\mathrm{m}})| \;  \; \frac{\mathrm{m}}{3} \left[ \left( \frac{1}{\mathrm{m}^3}\right) +\frac{3}{\mathrm{m}} \left(\frac{j}{\mathrm{m}}-\log z\right)^{2} + \frac{3}{\mathrm{m}^{2}}\left(\frac{j}{\mathrm{m}}-\log z\right)\right] \\ &  + \frac{\|h\|_{\mathrm{w}}}{\zeta_{z}}\bigvee_{j\in 
         \mathbb{Z}} |\varkappa(e^{-j}z^{\mathrm{m}})| 2\; \mathrm{m} \; |\log z| \left( \int_{\frac{j}{\mathrm{m}}}^{\frac{j+1}{\mathrm{m}}} |v-\frac{j}{\mathrm{m}}| dv +  \int_{\frac{j}{\mathrm{m}}}^{\frac{j+1}{\mathrm{m}}} |\frac{j}{\mathrm{m}}-\log z| dv\right) \\
\end{align*}

\begin{align*}I_{1}
    \leq & \frac{\|h\|_{\mathrm{w}}}{\zeta_{z}}\bigvee_{j\in 
         \mathbb{Z}} |\varkappa(e^{-j}z^{\mathrm{m}})| \;  \; \frac{\mathrm{m}}{3} \left[ \frac{1}{\mathrm{m}^3} +\frac{3}{\mathrm{m}} \left(\frac{|j-\mathrm{m}\log z|^{2}}{\mathrm{m}^{2}}\right) + \frac{3}{\mathrm{m}^{2}}\left(\frac{|j-\mathrm{m}\log z|}{\mathrm{m}}\right)\right] \\ &  + \frac{\|h\|_{\mathrm{w}}}{\zeta_{z}}\bigvee_{j\in \mathbb{Z}} |\varkappa(e^{-j}z^{\mathrm{m}})| 2\; \mathrm{m} |\log z| \left( \frac{1}{2\mathrm{m}^{2}}+ \frac{|j-\mathrm{m} \log z|}{\mathrm{m}^{2}}\right) \\ 
    \leq & \frac{\|h\|_{\mathrm{w}}}{\zeta_{z}}\bigvee_{j\in 
         \mathbb{Z}} |\varkappa(e^{-j}z^{\mathrm{m}})| \;  \; \left[ \frac{1}{3 \mathrm{m}^2} + \frac{|j-\mathrm{m}\log z|^{2}}{\mathrm{m}^{2}} + \frac{|j-\mathrm{m}\log z|}{\mathrm{m}^2}\right] \\ & + \frac{\|h\|_{\mathrm{w}}}{\zeta_{z}}\bigvee_{j\in 
         \mathbb{Z}} |\varkappa(e^{-j}z^{\mathrm{m}})| 2\; |\log z| \left( \frac{1}{2\mathrm{m}}+ \frac{|j-\mathrm{m}\log z|}{\mathrm{m}}\right) \\
    \leq & \frac{\|h\|_{\mathrm{w}}}{\zeta_{z}}\bigvee_{j\in
         \mathbb{Z}} |\varkappa(e^{-j}z^{\mathrm{m}})| \;  \; \left[ \frac{|j-\mathrm{m} \log z|^{2}}{\mathrm{m}^{2}} + \frac{|j-\mathrm{m}\log z|}{\mathrm{m}} \left( \frac{1}{\mathrm{m}}+ 2 |\log z|\right)+ \frac{1}{\mathrm{m}} \left( \frac{1}{3\mathrm{m}}+ |\log z|\right)\right]\\
    \Rightarrow I_{1} \leq & \frac{\|h\|_{\mathrm{w}}}{\zeta_{z}} \left[ \frac{m_2(\varkappa)}{\mathrm{m}^2} + \frac{m_1(\varkappa)}{\mathrm{m}} \left( \frac{1}{\mathrm{m}}+ 2|\log z|\right)+ \frac{m_0(\varkappa)}{\mathrm{m}} \left( \frac{1}{3\mathrm{m}}+|\log z|\right)  \right] .\\
\end{align*}
Let us consider $I_{2}$. Let $z \in \mathbb{R_{+}} \And{\epsilon > 0}$. Since $h$ is log-continuous at $z$, then $\mathrm{w}h$ is also log-continuous at $z$. Hence $\exists \;\rho > 0$ such that 
\begin{equation*}
\left| \mathrm{w}(e^{\frac{j}{\mathrm{m}}}) h(e^{\frac{j}{\mathrm{m}}}) -\mathrm{w}(z) h(z)\right| < \epsilon \; \text{whenever} \left|\frac{j}{\mathrm{m}} - \log(z)\right| \leq \rho.
\end{equation*}
Then, we can write 
\begin{align*}
    I_{2} = & \frac{1}{\zeta_{z} \mathrm{w}(z)}\bigvee_{j\in \mathbb{Z}} |\varkappa(e^{-j}z^{\mathrm{m}})|  \; \mathrm{m}\int_{\frac{j}{\mathrm{m}}}^{\frac{j+1}{\mathrm{m}}} |\mathrm{w}(e^{v})f(e^{v})-\mathrm{w}(z)h(z)| dv\\
    = & \frac{1}{\zeta_{z} \mathrm{w}(z)} \bigvee_{\substack{j\in\mathbb{Z} \\ |j-\mathrm{m}\log(z)|\leq \rho}} |\varkappa(e^{-j}z^{\mathrm{m}})|  \; \mathrm{m}\int_{\frac{j}{\mathrm{m}}}^{\frac{j+1}{\mathrm{m}}} |\mathrm{w}(e^{v})f(e^{v})-\mathrm{w}(z)h(z)| dv\\ & + \frac{1}{\zeta_{z} \mathrm{w}(z)} \bigvee_{\substack{j\in\mathbb{Z} \\ |j-\mathrm{m}\log(z)|> \rho}} |\varkappa(e^{-j}z^{\mathrm{m}})|  \; \mathrm{m}\int_{\frac{j}{\mathrm{m}}}^{\frac{j+1}{\mathrm{m}}} |\mathrm{w}(e^{v})f(e^{v})-\mathrm{w}(z)h(z)| dv\\
    = & \frac{1}{\zeta_{z} \mathrm{w}(z)}max\{I_{2.1},I_{2.2}\}.
\end{align*}
But
\begin{align*}
    I_{2.1} = & \bigvee_{\substack{j\in\mathbb{Z} \\ |j-\mathrm{m}\log(z)|\leq \rho}} |\varkappa(e^{-j}z^{\mathrm{m}})|  \; \mathrm{m}\int_{\frac{j}{\mathrm{m}}}^{\frac{j+1}{\mathrm{m}}} |\mathrm{w}(e^{v})f(e^{v})-\mathrm{w}(z)h(z)| dv \\
    \leq &  \; \epsilon \bigvee_{j\in\mathbb{Z}} |\varkappa(e^{-j}z^{\mathrm{m}})| \\
    \leq & \; \epsilon \; m_{0}(\varkappa).
\end{align*}
and
\begin{align*}
    I_{2.2} = & \bigvee_{\substack{j\in\mathbb{Z} \\ |j-\mathrm{m}\log(z)|> \rho}} |\varkappa(e^{-j}z^{\mathrm{m}})|  \; \mathrm{m}\int_{\frac{j}{\mathrm{m}}}^{\frac{j+1}{\mathrm{m}}} |\mathrm{w}(e^{v})f(e^{v})-\mathrm{w}(z)h(z)| dv\\
    = & 2\|h\|_{\mathrm{w}}\bigvee_{\substack{j\in\mathbb{Z} \\ |j-\mathrm{m}\log(z)|> \rho}} |\varkappa(e^{-j}z^{\mathrm{m}})|\\
    \leq & 2 \; \epsilon \; \|h\|_{\mathrm{w}}.
\end{align*}
Therefore, combining the estimates $I_{1}, I_{2.1}$ and ${I_{2.2}}$, we have
\begin{align*}
    |\mathscr{M}^{\varkappa}_{\mathrm{m}}(h,z) - h(z)| \leq  & \mathscr{M}^{\varkappa}_{\mathrm{m}}(|h-h_{z}|,z)\\
    \leq & \frac{\|h\|_{\mathrm{w}}}{\zeta_{z}} \left[ \frac{m_2(\varkappa)}{\mathrm{m}^2} + \frac{m_1(\varkappa)}{\mathrm{m}} \left( \frac{1}{\mathrm{m}}+ 2|\log z|\right)+ \frac{m_0(\varkappa)}{\mathrm{m}} \left( \frac{1}{3\mathrm{m}}+|\log z|\right)  \right] \\ & +\frac{1}{\zeta_{z} \mathrm{w}(z)}\max\{\epsilon \; m_{0}(\varkappa), 2 \; \epsilon \; \|h\|_{\mathrm{w}}\}.\numberthis \label{eqn}
\end{align*}
Taking limit both sides as $\mathrm{m} \to \infty$, we have 
\begin{equation*}
  \lim_{\mathrm{m} \to \infty } \mathscr{M}^{\varkappa}_{\mathrm{m}}(h,z) = h(z).
\end{equation*}
For $h \in \mathscr{V}^{\mathrm{w}}_{+}(\mathbb{R_{+}}) $ and from the inequality \eqref{eqn}, we get 
\begin{align*}
   |\mathscr{M}^{\varkappa}_{\mathrm{m}}(h,z) - h(z)|  \leq & \frac{\|h\|_{\mathrm{w}}}{\zeta_{z}} \left[ \frac{m_2(\varkappa)}{\mathrm{m}^2} + \frac{m_1(\varkappa)}{\mathrm{m}} \left( \frac{1}{\mathrm{m}}+ 2|\log z|\right)+ \frac{m_0(\varkappa)}{\mathrm{m}} \left( \frac{1}{3\mathrm{m}}+|\log z|\right)  \right] \\ & +\frac{1}{\zeta_{z} \mathrm{w}(z)}max\{\epsilon \; m_{0}(\varkappa), 2 \; \epsilon \; \|h\|_{\mathrm{w}}\} \\
    \Rightarrow \mathrm{w}(z)\;|\mathscr{M}^{\varkappa}_{\mathrm{m}}(h,z) - h(z)| \leq & \frac{\mathrm{w}(z) \|h\|_{\mathrm{w}}}{\zeta_{z}}\left[ \frac{m_2(\varkappa)}{\mathrm{m}^2} + \frac{m_1(\varkappa)}{\mathrm{m}} \left( \frac{1}{\mathrm{m}}+ 2|\log z|\right)+ \frac{m_0(\varkappa)}{\mathrm{m}} \left( \frac{1}{3\mathrm{m}}+|\log z|\right)  \right] \\ & +\frac{1}{\zeta_{z}} max\{\epsilon \; m_{0}(\varkappa), 2 \; \epsilon \; \|h\|_{\mathrm{w}}\}. 
\end{align*}
By taking supremum over $z \in \mathbb{R_{+}} \And{\mathrm{m} \to \infty}$, we have
\begin{equation*}
\lim_{\mathrm{m} \to \infty } \|\mathscr{M}^{\varkappa}_{\mathrm{m}}(h,z) - h(z)\|_{\mathrm{w}} = 0.
\end{equation*}
\end{proof}

\begin{thm}
    Let $\varkappa$ be a kernel satisfying the assumptions $(\varkappa_{1}) \And{(\varkappa_{2})} $ for $\kappa = 5$. Then $h \in \mathscr{V}^{\mathrm{w}}_{+}(\mathbb{R_{+}}),$ 
    \begin{equation*}
    \|\mathscr{M}^{\varkappa}_{\mathrm{m}}(h,z) - h(z)\|_{\mathrm{w}} \leq \frac{256 \Upsilon(h,\frac{1}{\mathrm{m}})}{\zeta_{z}} [m_{0}(\varkappa) + 4m_5(\varkappa)] .
    \end{equation*}
\end{thm}
\begin{proof}
    Let $z \in \mathbb{R_{+}}$ be fixed. From the definition of the sampling operators $\mathscr{M}^{\varkappa}_{\mathrm{m}}$ and using the definition of logarithmic modulus of continuity we can write 
    \begin{equation*}
     |\mathscr{M}^{\varkappa}_{\mathrm{m}}(h,z) - h(z)| \leq \frac{1}{\zeta_{z}} \bigvee_{j\in\mathbb{Z}} | \varkappa(e^{-j}z^{\mathrm{m}})|\; \mathrm{m} \int_{\frac{j}{\mathrm{m}}}^{\frac{j+1}{\mathrm{m}}}|h(e^{v})- h(z)|\; dv .
     \end{equation*}
    Since  for all $h \in \mathscr{V}^{\mathrm{w}}_{+}(\mathbb{R_{+}}),$
    \begin{align*}
      |(h(s)- h(z)|&\leq 16 (1+\rho^{2})^{2} (1+\log^{2}z) \left(1+ \frac{|\log s-\log z|^{5}}{\rho^{5}}\right) \Upsilon(h,\rho) \\
      \Rightarrow |(h(e^{v})- h(z)|&\leq 16 (1+\rho^{2})^{2} (1+\log^{2}z) \left(1+ \frac{|v-\log z|^{5}}{\rho^{5}}\right) \Upsilon(h,\rho).
    \end{align*}

    So for any positive $\rho \leq 1$, we have
    \begin{align*}
	|\mathscr{M}^{\varkappa}_{\mathrm{m}}(h,z) - h(z)| 
        \leq &  \frac{1}{\zeta_{z}}\bigvee_{j\in \mathbb{Z}} \varkappa(e^{-j}z^{\mathrm{m}})| \; \mathrm{m} \; 
                \int_{\frac{j}{\mathrm{m}}}^{\frac{j+1}{\mathrm{m}}} 64 (1+\log^{2}z) \left(1+ \frac{|v-\log z|^{5}}{\rho^{5}}\right) \Upsilon(h,\rho)\; dv\\
	\leq &  \frac{64(1+\log^{2}z)\Upsilon(h,\rho)}{\zeta_{z}} \bigvee_{j\in \mathbb{Z}} | \varkappa(e^{- 
                j}z^{\mathrm{m}})| \mathrm{m} \int_{\frac{j}{\mathrm{m}}}^{\frac{j+1}{\mathrm{m}}}\left(1+ \frac{|v-\log z|^{5}}{\rho^{5}}\;  dv \right)\\
        \leq &  \frac{64(1+\log^{2}z)\Upsilon(h,\rho)}{\zeta_{z}} \bigvee_{j\in \mathbb{Z}} | \varkappa(e^{- 
                 j}z^{\mathrm{m}})|\; \mathrm{m} \; \left[ \frac{1}{\mathrm{m}}+\int_{\frac{k}{\mathrm{m}}}^{\frac{j+1}{\mathrm{m}}}\frac{|v-\log z|^{5}}{\rho^{5}} \; dv\right]\\
        \leq &  \frac{64(1+\log^{2}z)\Upsilon(h,\rho)}{\zeta_{z}} \bigvee_{j\in \mathbb{Z}} | \varkappa(e^{- 
                j}z^{\mathrm{m}})| \\ & \mathrm{m} \; \left[ \frac{1}{\mathrm{m}}+\frac{2^4}{\rho^{5}}\int_{\frac{j}{\mathrm{m}}}^{\frac{j+1}{\mathrm{m}}} \left(\left|v-\frac{j}{\mathrm{m}}\right|^{5}+ \left|\frac{j}{\mathrm{m}} -\log z \right|^{5}\right)\;  dv\right]\\
        \leq & \frac{64(1+\log^{2}z)\Upsilon(h,\rho)}{\zeta_{z}} \bigvee_{j\in \mathbb{Z}} | \varkappa(e^{- 
                j}z^{\mathrm{m}})|\; \mathrm{m} \; \left[ \frac{1}{\mathrm{m}}+\frac{2^4}{\rho^{5}} \left(\frac{1}{6\mathrm{m}^{6}}+ \frac{\left|j-\mathrm{m}\log z\right|^{5}}{\mathrm{m}^{6}}\right) \right]\\
        \leq & \frac{64(1+\log^{2}z)\Upsilon(h,\rho)}{\zeta_{z}} \bigvee_{j\in \mathbb{Z}} | \varkappa(e^{- 
                j}z^{\mathrm{m}})| \left[ 1 +\frac{2^4}{6\mathrm{m}^{5}\rho^{5}} +\frac{2^4}{\mathrm{m}^{5}\rho^{5}} \left|j-\mathrm{m}\log z\right|^{5}\right]\\
	\leq & \frac{64(1+\log^{2}z) \Omega(h,\rho)}{\zeta_{z}}\left[ \left( 1 +\frac{2^4}  
               {6\mathrm{m}^{5}\rho^{5}}\right)m_{0}(\varkappa) + \frac{16}{\mathrm{m}^5 \rho^{5}} m_5{(\varkappa)}\right] 
    \end{align*}
    Finally choosing $\rho = \frac{1}{\mathrm{m}}$ $\mathrm{m} \geq 1$  we have 
    \begin{align*}
        |\mathscr{M}^{\varkappa}_{\mathrm{m}}(h,z) - h(z)| 
        \leq & \frac{64(1+\log^{2}z) \Upsilon(h,\rho)}{\zeta_{z}}\left[4m_{0}(\varkappa) + 16 m_5{\varkappa}\right] \\
        \leq & \frac{256(1+\log^{2}z) \Upsilon(h,\rho)}{\zeta_{z}}\left[ m_{0}(\varkappa) + 4 m_5{\varkappa}\right] 
    \end{align*}
    
    which is desired.
\end{proof}

Let $ k  \in \mathbb{N} $, then  the algebraic moment of order $k$ of a kernel defined by 
$$\mathcal{A}_{j}(\varkappa,s) =  \bigvee_{j \in \mathbb{Z}} \varkappa(e^{-j}s) (j-\log s)^{k} \; \; \text{$\forall z \in \mathbb{R}_{+}$} .$$

In order to present Voronovskaja theorem in quantitative form, we need a further assumption on kernel function $\varkappa\; $, i.e., there exists $ n \in \mathbb{N}$ such that for every $k \in \mathbf{N}_{0}$, $j\leq n$ there holds:\\
$(\varkappa_{3} ): \mathcal{A}_{k}(\varkappa,z) = \mathcal{A}_{k}(\varkappa)\in \mathbb{R}_{+} $\;  is independent of $z$.\\

Let us consider the Mellin-Taylor formula given in \cite{mamedov2003,bardaro2020}. For any $h\in \mathscr{V}^{\mathrm{w}}_{+}(\mathbb{R_{+}})$ belonging to $\mathscr{C}^{n}(\mathbb{R}_{+})$ locally at the point $z\in \mathbb{R}_{+}$, the Mellin formula with the Mellin derivatives is defined by 

$$ h(r) = \sum_{r=0}^{n} \frac{\theta^{r}h(z)}{r!} (\log s -\log z)^{r} + \mathscr{R}_n(h;s,z),$$
where $$ \mathscr{R}_{n}(h;s,z) = \frac{\theta^{n}h(\vartheta) -\theta^{n}f(z) }{n!}(\log s-\log z)^{n} $$
is the Lagrange remainder in Mellin-Taylor's formula at the point $z\in \mathbb{R}_{+}$ and $\vartheta $ is a number lying  between $s$ and $z$. According the inequality 
$$|(h(s)- h(z)|\leq 16 (1+\rho^{2})^{2} (1+\log^{2}z) (1+ \frac{|\log s-\log z|^{5}}{\rho^{5}}) \Upsilon(h,\rho)$$ with the similar method presented in \cite{bardaro9}, we can easily have the estimate 
$$\left| \mathscr{R}_n(h;s,z) \right| \leq  \frac{64}{n!} (1+\log^{2}z) \Upsilon(\theta^{n}h,\rho) \left( |\log s-\log z|^{n} + \frac{|\log s-\log z|^{n+5}}{\rho^{5}}\right).$$
\begin{thm}
    Let $\varkappa$ be the kernel satisfying the assumptions  $(\varkappa_{1})$  $(\varkappa_{2})$ and $(\varkappa_{3})$ for $\kappa = n+5 $  $ n \in \mathbb{N}$ and in addition $(\varkappa_{3})$  is satisfied   for every $ z \in \mathbb{R}_{+}.$If $h^{n} \in \mathscr{V}^{\mathrm{w}}_{+}(\mathbb{R_{+}})$ for $ h \in \mathscr{C}^{n}(\mathbb{R}_{+}) $ then we have for every $z \in \mathbb{R}_{+}$ that
    \begin{align*}
      &\displaystyle \left|\mathrm{m}[(\mathscr{M}^{\varkappa}_{\mathrm{m}}h)(z) -h(z)] -\frac{1}{\mathcal{A}_{0}(\varkappa)} \sum_{r=1}^{n} \frac{\theta^{r}f(z)}{r! \mathrm{m}^{r-1}} \sum_{l=1}^{r} \binom{r}{l}  \frac{\mathcal{A}_{l}(\varkappa)}{(r-l+1)}\right| \\
      &\leq  \frac{2^{n+5}}{\mathcal{A}_{0}(\varkappa)\;\mathrm{m}^{n-1} n!}(1+\log^{2}z) \Upsilon(\theta^{n} h,\mathrm{m}^{-1}) \left[ \frac{m_{0}(\varkappa)}{n+1} +\frac{ 2^{5} m_{0}(\varkappa)}{n+6}+ m_{n} (\varkappa) + 2^{5}m_{n+5} (\varkappa)  \right] .
    \end{align*}
\end{thm}
	
\begin{proof}
Since $h^{n} \in \mathscr{V}^{\mathrm{w}}_{+}(\mathbb{R_{+}})$ , using Mellin-Taylor expansion at a point $z \in \mathbb{R}_{+}$ and by definition of $(\mathscr{M}^{\varkappa}_{\mathrm{m}}h)$ we can write.
\begin{align*}
    (\mathscr{M}^{\varkappa}_{\mathrm{m}}h) \leq & \frac{1}{\mathcal{A}_{0}(\varkappa)}\bigvee_{j\in \mathbb{Z}} \varkappa(e^{-j}z^{\mathrm{m}}) \; \mathrm{m} \int_{\frac{j}{\mathrm{m}}}^{\frac{j+1}{\mathrm{m}}} \sum_{k=0}^{n} \frac{\theta^{k}f(z)}{k!} (v-\log z)^{k} dv \\ & + \frac{1}{\mathcal{A}_{0}(\varkappa)}\bigvee_{j\in \mathbb{Z}}  \varkappa(e^{-j}z^{\mathrm{m}})\; \mathrm{m} \int_{\frac{j}{\mathrm{m}}}^{\frac{j+1}{\mathrm{m}}} \mathscr{R}_n(h;v,z) dv \\
    = & I_{1} + I_{2}	
\end{align*}
Let us first consider $I_1.$ In view of binomial expansion , we obtain 
\begin{align*}
     I_{1}=&  \frac{1}{\mathcal{A}_{0}(\varkappa)}\sum_{k=0}^{n} \frac{\theta^{k}h(z)}{k!}  \bigvee_{j\in \mathbb{Z}} \varkappa(e^{-j}z^{\mathrm{m}}) \mathrm{m} \int_{\frac{j}{\mathrm{m}}}^{\frac{j+1}{\mathrm{m}}} \sum_{l=0}^{k} \binom{k}{l} \left(v-\frac{j}{\mathrm{m}}\right)^{k-l} \left(\frac{j}{\mathrm{m}} - \log z\right)^{l} dv \\
     = &  \frac{1}{\mathcal{A}_{0}(\varkappa)} \sum_{k=0}^{n} \frac{\theta^{k}h(z)}{k!} \sum_{l=0}^{k} \binom{k}{l} \mathrm{m} \; \bigvee_{j\in \mathbb{Z}} \varkappa(e^{-j}z^{\mathrm{m}}) \frac{(j-\mathrm{m}\log z)^l}{\mathrm{m}^{l}} \int_{\frac{j}{\mathrm{m}}}^{\frac{j+1}{\mathrm{m}}}\left(v-\frac{j}{\mathrm{m}}\right)^{k-l} dv\\
     = &  \frac{1}{\mathcal{A}_{0}(\varkappa)} \sum_{k=0}^{n} \frac{\theta^{k}h(z)}{k!} \sum_{l=0}^{k} \binom{k}{l} \frac{1}{\mathrm{m}^{l-1}} \bigvee_{j\in \mathbb{Z}} \varkappa(e^{-j}z^{\mathrm{m}}) (j-\mathrm{m}\log z)^l \frac{1}{(k-l+1) \mathrm{m}^{k-l+1}} \\
     = & \frac{1}{\mathcal{A}_{0}(\varkappa)} \sum_{k=0}^{n} \frac{\theta^{k}h(z)}{k!} \sum_{l=0}^{k} \binom{k}{l} \frac{1}{\mathrm{m}^{l-1}}  \frac{1}{(k-l+1) \mathrm{m}^{k-l+1}} \bigvee_{j\in \mathbb{Z}} \varkappa(e^{-j}z^{\mathrm{m}}) (j-\mathrm{m}\log z)^l \\
     = &  \frac{1}{\mathcal{A}_{0}(\varkappa)} \sum_{k=0}^{n} \frac{\theta^{k}h(z)}{k!} \sum_{l=0}^{k} \binom{k}{l}   \frac{1}{(k-l+1)\mathrm{m}^{k}} \;  \mathcal{A}_{l}(\varkappa,z)\\  
     = & \frac{1}{\mathcal{A}_{0}(\varkappa)} \sum_{k=0}^{n} \frac{\theta^{k}h(z)}{k!\; \mathrm{m}^{k}} \sum_{l=0}^{k} \binom{k}{l}   \frac{1}{(k-l+1)} \;  \mathcal{A}_{l}(\varkappa,z)\\
     = & h(z) +  \frac{1}{\mathcal{A}_{0}(\varkappa)}\frac{1}{\mathrm{m}}\;\sum_{k=1}^{n} \frac{\theta^{k}h(z)}{k! \mathrm{m}^{k-1}} \sum_{l=0}^{k} \binom{k}{l} \frac{1}{k-l+1} \mathcal{A}_{l}(\varkappa,z).
\end{align*}
	
On the other-hand concerning $I_{2}$ by the inequality
$$\left| \mathscr{R}_n(h;v,z) \right| \leq  \frac{64}{n!} (1+\log^{2}z) \Upsilon(\theta^{n}h,\rho) \left( |v-\log z|^{n} + \frac{|v-\log z|^{n+5}}{\rho^{5}}\right),$$ we have
\begin{align*}
    |I_{2}| \leq & \frac{1}{\mathcal{A}_{0}(\varkappa)}\frac{64\; \mathrm{m}}{n!} (1+\log^{2}z) \Upsilon(\theta^{n}h,\rho) \bigvee_{j\in \mathbb{Z}} \varkappa(e^{-j}z^{\mathrm{m}}) \int_{\frac{j}{\mathrm{m}}}^{\frac{j+1}{\mathrm{m}}} \left( |v-\log z|^{n} + \frac{|v-\log z|^{n+5}}{\rho^{5}}\right) dv\\
    \leq & \frac{1}{\mathcal{A}_{0}(\varkappa)}\frac{64\; \mathrm{m}}{n!} (1+\log^{2}z) \Upsilon(\theta^{n}h,\rho) \bigvee_{j\in \mathbb{Z}} \varkappa(e^{-j}z^{\mathrm{m}})\\ & \left[ \int_{\frac{j}{\mathrm{m}}}^{\frac{j+1}{\mathrm{m}}}  \left|\left(v- \frac{j}{\mathrm{m}}\right) + \left( \frac{j}{\mathrm{m}} - \log z \right)\right|^{n} dv + \frac{1}{\rho^{5}}\int_{\frac{j}{\mathrm{m}}}^{\frac{j+1}{\mathrm{m}}} \left|\left(v- \frac{j}{\mathrm{m}}\right) + \left( \frac{j}{\mathrm{m}} - \log z \right)\right|^{n+5}  dv \right]\\
    \leq & \frac{1}{\mathcal{A}_{0}(\varkappa)}\frac{64\; \mathrm{m}}{n!} (1+\log^{2}z) \Upsilon(\theta^{n}h,\rho) \bigvee_{j\in \mathbb{Z}} \varkappa(e^{-j}z^{\mathrm{m}})\\ & \left[  2^ {n-1} \int_{\frac{j}{\mathrm{m}}}^{\frac{j+1}{\mathrm{m}}} \left( \left|v - \frac{j}{\mathrm{m}}\right|^{n} + \left| \frac{j}{\mathrm{m}} - \log z \right|^{n} \right) dv + \frac{2^{n+4}}{\rho^{5}}\int_{\frac{j}{\mathrm{m}}}^{\frac{j+1}{\mathrm{m}}} \left(\left|v- \frac{j}{\mathrm{m}}\right|^{n+5} + \left| \frac{j}{\mathrm{m}} - \log z \right|^{n+5} \right) dv \right] \\
    \leq & \frac{1}{\mathcal{A}_{0}(\varkappa)}\frac{64\; \mathrm{m}}{n!} (1+\log^{2}z) \Upsilon(\theta^{n}h,\rho) \bigvee_{j\in \mathbb{Z}} \varkappa(e^{-j}z^{\mathrm{m}})\\ & \left[  2^ {n-1}\left( \frac{1}{(n+1)\mathrm{m}^{n+1}} + \frac{|j-\mathrm{m}\log z|^{n}}{\mathrm{m}^{n+1}} \right)  + \frac{2^{n+4}}{\rho^{5}} \left( \frac{1}{(n+6)\mathrm{m}^{n+6}} + \frac{|j-\mathrm{m}\log z|^{n+5}}{\mathrm{m}^{n+6}} \right) \right] \\
    \leq & \frac{1}{\mathcal{A}_{0}(\varkappa)}\frac{64\; \mathrm{m}\; 2^{n-1}}{n! \mathrm{m}^{n+1}} (1+\log^{2}z) \Upsilon(\theta^{n}h,\rho) \bigvee_{j\in \mathbb{Z}} \varkappa(e^{-j}z^{\mathrm{m}})\\ & \left[  \left( \frac{1}{(n+1)} + |j-\mathrm{m}\log z|^{n} \right)  + \frac{2^{5}}{\delta^{5} \mathrm{m}^{5}} \left( \frac{1}{(n+6)} + |j-\mathrm{m}\log z|^{n+5} \right) \right] \\
    \leq & \frac{ 2^{n+5}}{\mathcal{A}_{0}(\varkappa) \; n! \mathrm{m}^{n}} (1+\log^{2}z) \Upsilon(\theta^{n}h,\rho)  \\ &\left[  \left( \frac{1}{(n+1)} m_{0}(\varkappa) + m_{n}(\varkappa) \right)  + \frac{2^{5}}{\rho^{5} \mathrm{m}^{5}} \left( \frac{1}{(n+6)} m_{0}(\varkappa) + m_{n+5}(\varkappa) \right) \right]. \\
\end{align*}\\
Now choosing $\rho = \mathrm{m}^{-1}$
\begin{align*}
    |I_{2}| \leq & \frac{ 2^{n+5}}{\mathcal{A}_{0}(\varkappa) \; n! \mathrm{m}^{n}} (1+\log^{2}z) \Upsilon(\theta^{n}h,\rho) \\ & \left[  \left( \frac{1}{(n+1)} m_{0}(\varkappa) + m_{n}(\varkappa) \right)  + 2^{5}  \left( \frac{1}{(n+6)} m_{0}(\varkappa) + m_{n+5}(\varkappa) \right) \right].
\end{align*}
This proves the result.
\end{proof}
\section{Examples of Kernels and  Graphical Visualizations}\label{section4}
In this section we present few examples of kernels satisfying the assumptions given in discussed theory. To visualize the effectiveness of the Max-Product Kantorovich-type Exponential Sampling Operators, we present graphical visualizations for a range of test functions and kernel selections over the interval $(0.1, 10)$. Specifically, we consider the following test functions:
$$h_1(z) = e^{-z} \cos(2\pi z), \quad h_2(z) = \log(1+z), \quad h_3(z) = \frac{\sin(z)}{1+z^2},$$
along with three well-known kernel functions: the B-spline kernel, the Mellin--Fejér kernel, and the Mellin--Jackson kernel (see \cite{ANG}).

\subsection{B-spline Kernel}

The B-spline kernels in the Mellin framework are defined for $z \in \mathbb{R}^{+}$ by:
\begin{equation*}
    B_n(z) := \frac{1}{(n-1)!} \sum_{k=0}^{n} (-1)^k \binom{n}{k} \left(\frac{n}{2} + \log z - k\right)_{+}^{n+1},
\end{equation*}
where $(x)_+ := \max\{x, 0\}$. The compact support of $B_n$ in $\mathbb{R}^{+}$ ensures that the kernel meets the necessary conditions $(\varkappa_1)$ and $(\varkappa_2)$ for the Kantorovich-type operators.

We begin by analyzing the performance of the operator for the test function $h_1(z) = e^{-z} \cos(2\pi z)$ using the B-spline kernel of order $3$. The kernel used is given by:
\begin{equation*}
B_3(z) = \begin{cases}
    \frac{3}{4} - z^2, & \text{if } |z| \leq 1, \\
    \frac{1}{2}(2 - |z|)^2, & \text{if } 1 < |z| \leq 2, \\
    0, & \text{otherwise.}
\end{cases}
\end{equation*}

\begin{figure}[h!]
  \centering
  \includegraphics[width=1\linewidth]{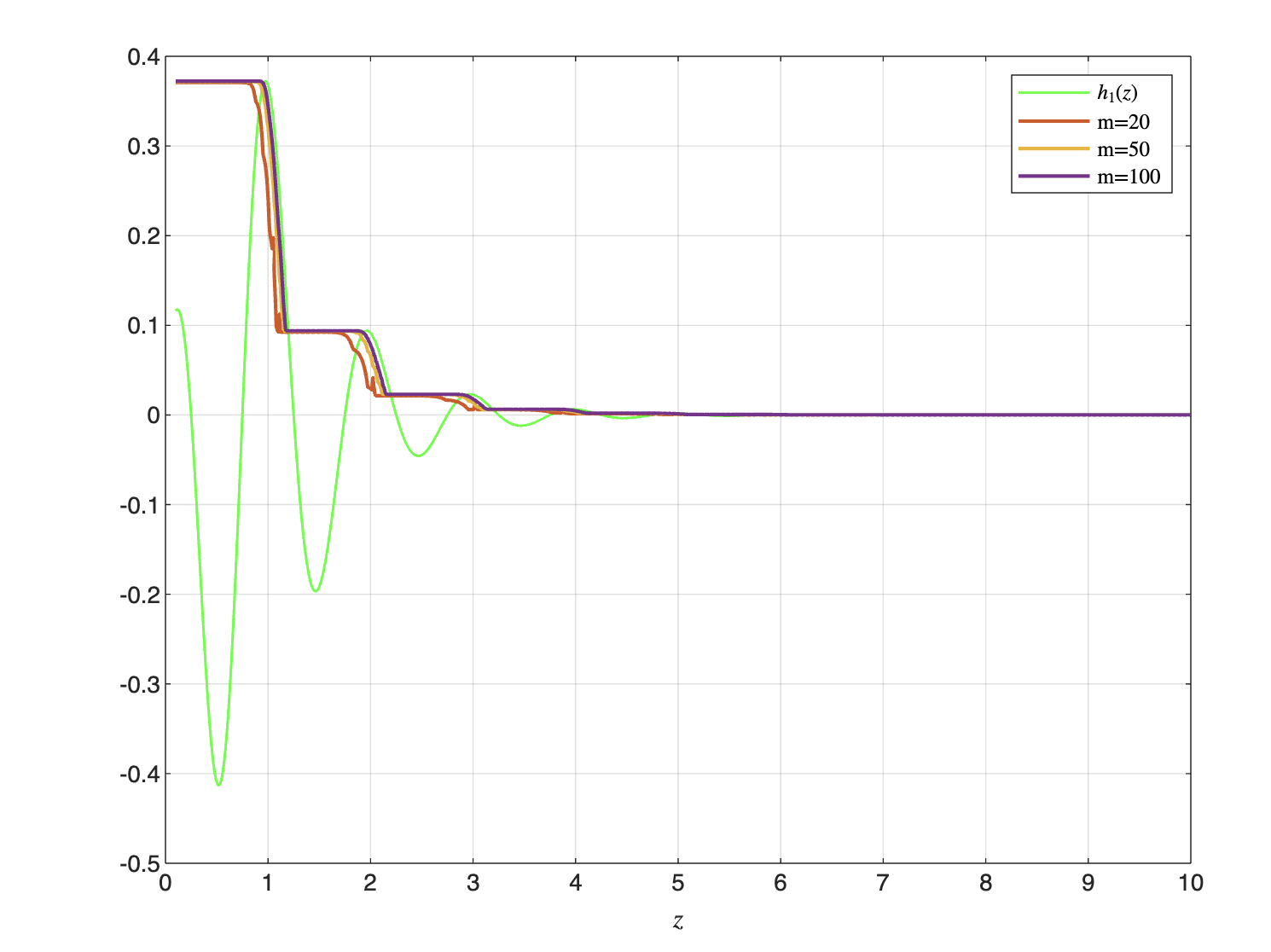}
  \caption{Approximation of $h_{1}(z)$ by $\mathscr{M}^{B_{3}}_{m}(h,z)$ for $m=20,50, 100$ .}
  \label{fig:mellin-jackson-h3}
\end{figure}

\vspace{1em}

\noindent\parbox{\linewidth}{
\textbf{Table 1.} Estimation of the approximation error (up to 4 decimal points)  in the approximation of $h_{1}(z)$ by $\mathscr{M}^{B_{3}}_{m}(h,z)$ for $ m = 20,50, 100 $.
\vspace{0.3em}
}

\begin{table}[htbp]
\centering
\begin{tabular}{c|c|cc|cc|cc}
\hline
$z$ & Exact & \multicolumn{2}{c|}{$m=20$} & \multicolumn{2}{c|}{$m=50$} & \multicolumn{2}{c}{$m=100$} \\
    &       & Approx & Error & Approx & Error & Approx & Error \\
\hline
0.50 & -0.4097 & 0.3707 & 0.7804 & 0.3718 & 0.7815 & 0.3723 & 0.7819 \\
1.00 &  0.3679 & 0.2437 & 0.1242 & 0.3224 & 0.0455 & 0.3452 & 0.0226 \\
2.00 &  0.0914 & 0.0293 & 0.0621 & 0.0662 & 0.0252 & 0.0786 & 0.0128 \\
4.00 &  0.0063 & 0.0014 & 0.0049 & 0.0030 & 0.0033 & 0.0046 & 0.0016 \\
8.00 &  0.0001 & 0.0000 & 0.0001 & 0.0000 & 0.0000 & 0.0000 & 0.0000 \\
\hline
\end{tabular}
\end{table}

\subsection{Mellin--Fej\'{e}r Kernel}

The Mellin--Fej\'{e}r kernel is defined for $z \in \mathbb{R}^{+}$, $t \in \mathbb{R}$, and $\beta \geq 1$ as:
\begin{equation*}
    F_{\beta}^{t}(z) = \frac{\beta}{2\pi z^t} \left[ \text{sinc}\left(\frac{\beta \pi \log z}{1} \right) \right]^2,
\end{equation*}
where $\text{sinc}(x) = \frac{\sin(\pi x)}{\pi x}$. For our illustration, we set $\beta=1$ and $t=0$.

Using this kernel, we approximate $h_2(z) = \log(1+z)$ with increasing values of $m$. The results are shown below.

\begin{figure}[h!]
  \centering
  \includegraphics[width=1\linewidth]{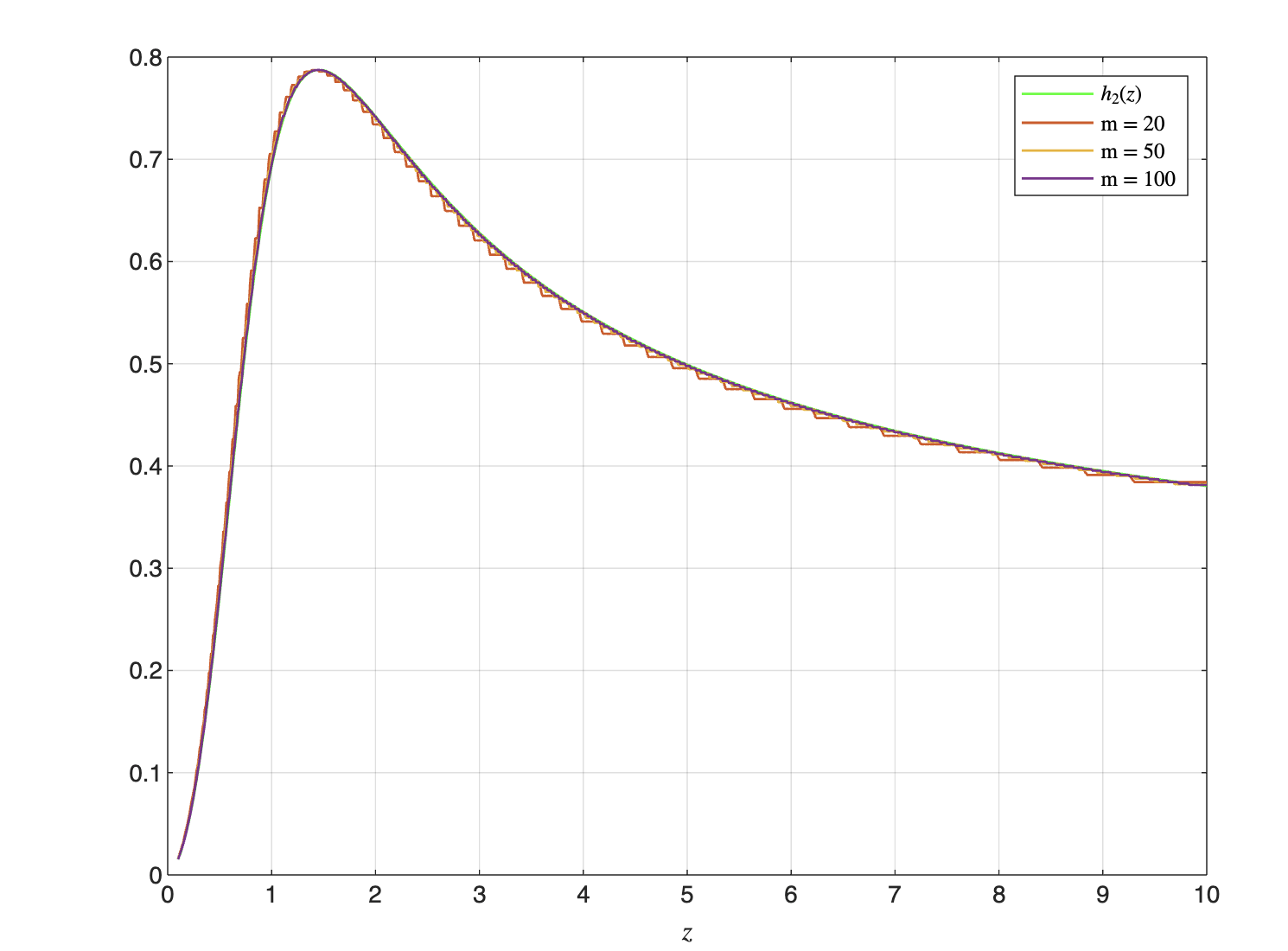}
  \caption{Approximation of $h_{2}(z)$ by $\mathscr{M}^{ F_{1}^{0}}_{m}(h,z)$ for $m=20,50,100$.}
  \label{fig:mellin-jackson-h3}
\end{figure}

\vspace{1em}
\noindent\parbox{\linewidth}{
\textbf{Table 2.} Approximation Error estimation (upto 4 decimal points)  in the approximation of $h_{2}(z)$ by $\mathscr{M}^{ F_{1}^{0}}_{m}(h,z)$ for $ m = 20,50, 100$.
\vspace{0.3em}
}

\begin{table}[htbp]
\centering
\begin{tabular}{c|c|cc|cc|cc}
\hline
$z$ & Exact & \multicolumn{2}{c|}{$m=20$} & \multicolumn{2}{c|}{$m=50$} & \multicolumn{2}{c}{$m=100$} \\
    &       & Approx & Error & Approx & Error & Approx & Error \\
\hline
0.50 & 0.2739 & 0.2901 & 0.0163 & 0.2778 & 0.0039 & 0.2778 & 0.0039 \\
1.00 & 0.6931 & 0.7052 & 0.0120 & 0.6981 & 0.0049 & 0.6956 & 0.0025 \\
2.00 & 0.7421 & 0.7341 & 0.0080 & 0.7379 & 0.0042 & 0.7416 & 0.0005 \\
4.00 & 0.5508 & 0.5413 & 0.0095 & 0.5499 & 0.0009 & 0.5487 & 0.0022 \\
8.00 & 0.4127 & 0.4073 & 0.0054 & 0.4111 & 0.0016 & 0.4118 & 0.0009 \\
\hline
\end{tabular}
\end{table}

\subsection{Mellin--Jackson Kernel}

Lastly, we consider the Mellin--Jackson kernel, defined as follows for $\beta \geq 1$, $n \in \mathbb{N}$, and $z \in \mathbb{R}^{+}$:
\begin{equation*}
    J^{t}_{\beta,n}(z) = C_{\beta,n} z^{-t} \left[ \text{sinc}\left(\frac{\log(z)}{2\beta \pi n} \right) \right]^{2n},
\end{equation*}
where $C_{\beta,n}^{-1} := \int_0^\infty \left[ \text{sinc}\left(\frac{\log(z)}{2\beta \pi n} \right) \right]^{2n} \mathrm{d}z$ and $\text{sinc}(v):= \begin{cases}
    \frac{sin\; \pi v}{\pi v} & v \neq 0\\
    1 &  v= 0
\end{cases} $.\\
We choose $\beta = 1$, $n = 3$, and $t = 0$ in our experiments.

The test function considered here is $h_3(z) = \frac{\sin(z)}{1 + z^2}$. The results demonstrate a clear improvement in the approximation quality as $m$ increases.

\begin{figure}[h!]
  \centering
  \includegraphics[width=1\linewidth]{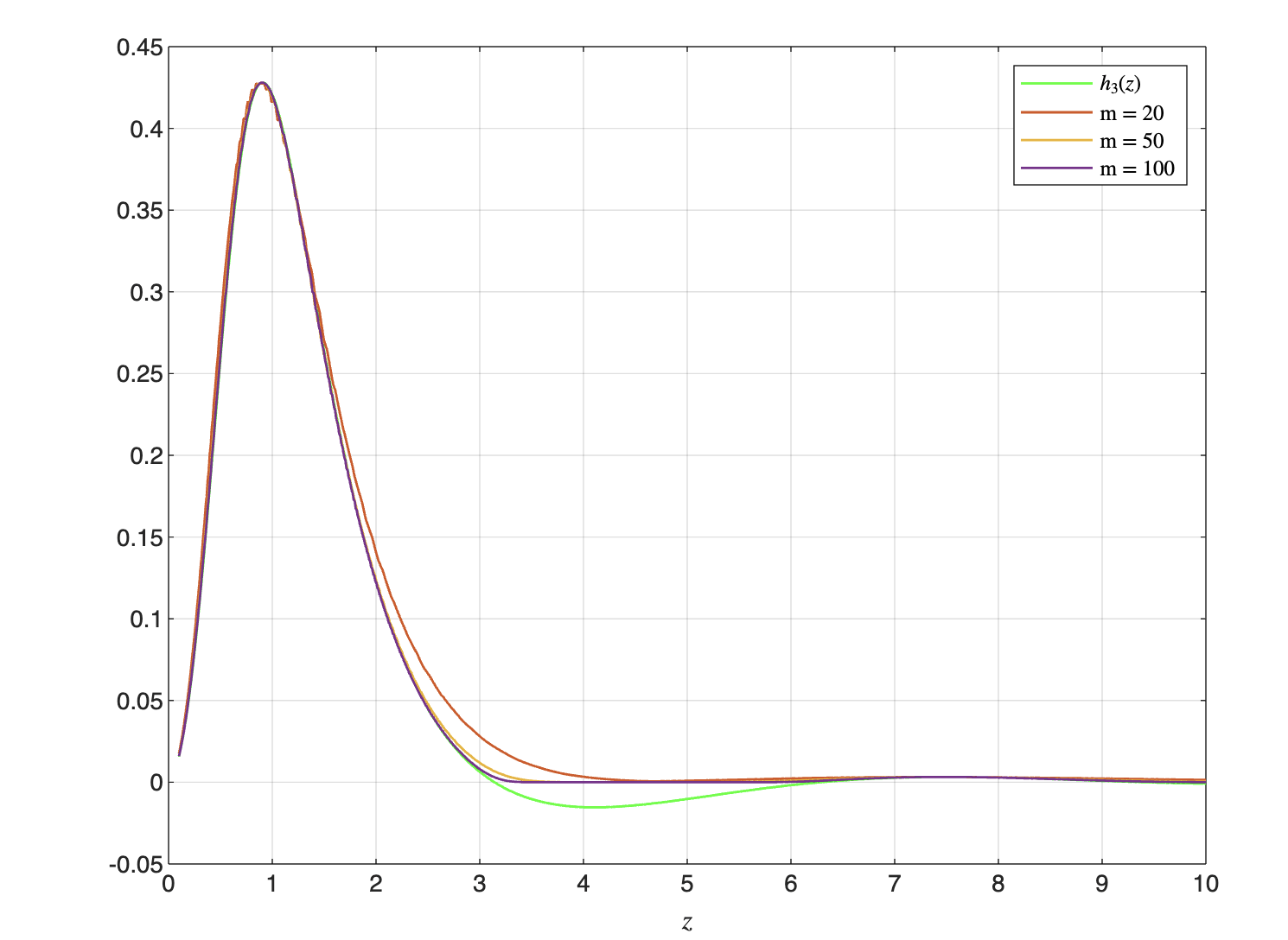}
  \caption{Approximation of $h_{3}(z)$ by  $\mathscr{M}^{ J^{0}_{1,3}}_{m}(h,z)$ for $m=20,50, 100$.}
  \label{fig:mellin-jackson-h3}
\end{figure}

\vspace{1em}

\noindent\parbox{\linewidth}{
\textbf{Table 3.}Error estimation (upto 4 decimal points) in the approximation of $h_{3}(z)$ by  $\mathscr{M}^{ J^{0}_{1,3}}_{m}(h,z)$ for $m=20,50, 100$.
\vspace{0.3em}
}

\begin{table}[htbp]
\centering
\begin{tabular}{c|c|cc|cc|cc}
\hline
$z$ & Exact & \multicolumn{2}{c|}{$m=20$} & \multicolumn{2}{c|}{$m=50$} & \multicolumn{2}{c}{$m=100$} \\
    &       & Approx & Error & Approx & Error & Approx & Error \\
\hline
0.50 & 0.2591 & 0.2794 & 0.0203 & 0.2655 & 0.0064 & 0.2621 & 0.0031 \\
1.00 & 0.4207 & 0.4163 & 0.0044 & 0.4191 & 0.0016 & 0.4200 & 0.0008 \\
2.00 & 0.1228 & 0.1407 & 0.0179 & 0.1234 & 0.0006 & 0.1221 & 0.0008 \\
4.00 & -0.0152 & 0.0033 & 0.0186 & 0.0000 & 0.0152 & 0.0000 & 0.0152 \\
8.00 & 0.0029 & 0.0030 & 0.0001 & 0.0029 & 0.0000 & 0.0028 & 0.0000 \\
\hline
\end{tabular}
\end{table}

\section{Conclusions}\label{section5}
In this paper, we studied the approximation properties of the Max-Product Kantrovich Exponential Sampling operator for functions in weighted spaces. We have proven the well-definedness of the operator and discuss its convergence, pointwise and in uniform sense. The weighted logarithmic modulus of continuity has been used to estimate the rate of convergence. Finally, we find that the choice of kernel function is a crucial factor in determining the accuracy of the approximation. We can get precise approximations for a large class of functions with an appropriate choice of the kernel and adjustments of the operator’s parameters.

\subsection*{Acknowledgments}
	We are extremely thankful to the reviewers for their careful reading of the manuscript and making valuable suggestions which led to a better presentation of the paper.
	\subsection*{Author Contributions} 
	All authors contributed equally to this work.

\subsection*{Conflicts of Interest}
	There is no conflict of interest regarding the publication of this article

\end{document}